\documentclass{article}
\usepackage{amsthm}

\usepackage{amsmath}
\usepackage{amssymb}
\newtheorem{theorem}{Theorem}
\newtheorem{lemma}{Lemma}
\usepackage{hyperref}
\newtheorem{definition}{Definition}
\usepackage{xcolor}
\pagecolor{white}
\color{black}

\usepackage{setspace}
\usepackage{amsmath, amssymb}

\begin{document}

\title{Optimal control problem for a nonlinear nonlocal evolution system describing  an interacting ternary mixture with an evaporating component: 2D case with bulk evaporation}
\author{Arghya Kundu \quad Adrian Muntean
\footnote{Both authors are with the Department of Mathematics and Computer Science, Karlstad University, Sweden}}
\date{\today}

\maketitle
\begin{abstract}
    We present an optimal control problem to guide the selection of morphology classes arising in organic solar cells. The study focuses on phase-separation processes in polymer–solvent mixtures, with particular attention to solvent evaporation as a mechanism to arrest morphology formation. We establish the existence of optimal controls  and analyze the Fréchet derivative of the control to state mapping. Finally, we derive the first-order necessary optimality condition via the corresponding adjoint system. 
\end{abstract}
\section{Introduction} 
We report here our first preliminary results concerning optimal control formulations aimed as tools to help deciding on classes of formed morphologies\footnote{Morphologies  are geometric pathways for electrons, charges, and so on, to move through the internal structure of the thin layer forming the solar cell, from anode to cathode.} as they occur in the framework of organic solar cells (OSC). The reader is referred to the review paper \cite{Review_Paper} for more details on the technological application and involved physical mechanisms. Although we nearly always have in mind the OSC setting, the reader should note that quite similar phase-separation scenarios are encountered in very different materials science applications. We mention here for instance the case when one attempts to influence the microstructural design of adhesive bands using lattice-based models for mixtures of (stochastically) interacting particle systems; see here the details \cite{Nils}.

Returning the discussion to OSCs, the nice particular feature of our setting is in fact twofold: (i) two components of the mixture are polymers and one is a solvent and (ii) the solvent evaporates - this fact is used to arrest the formation of the morphology (i.e. to stop the phase separation at a fixed concentration of evaporated solvent). We refer the reader to our numerical results reported in  \cite{CirilloLyonsMunteanEtAl2025MultiscaleandMultiphysicsModellingforAdvancedandSustainableMaterials} where we discussed evaporation matters in 2D; see {\em loc. cit.} for more references around this topic.  

In the framework of this article, we  aim to use the evaporation mechanism to formulate a first optimal control problem that allows for a certain design of solvent  distribution (linked in Section \ref{model_equations} to a given function $\phi_d\in L^2(S\times \Omega)$) to be achieved\footnote{This work should be seen as a proof-of-concept that the posed optimal control question can in fact be addressed. We present here only analysis results; serious computations are left for handling a  more realistic physical scenario.}. To this end, we use the model equations listed in   \cite{lyons2024phase} and consider the case when the evaporation mechanism acts as a bulk production term. The spirit of the calculations shown here follow to a large extent the way of working from \cite{frigeri2016optimal}.  Notice as well that two further cases of interest become apparent: (1) the evaporation takes place across part of the boundary of our domain (hence, boundary controls will need to be involved), and (2) the inverse temperature $\beta$ is used as control parameter. These cases will be studied elsewhere.

\subsection{Model equations}\label{model_equations}
Let $\Omega \subset \mathbb{R}^2$ be a bounded rectangle with the boundary $\partial \Omega$, and set a $T >0$ as a fixed time.  $S:=(0,T)$ is the time interval where we define our problem. Furthermore, we take $ \theta_{min}, \theta_{max} \in \mathbb{R}$ with $\theta_{min} \leq \theta_{max}$ and define the set 
\begin{equation}
    \mathcal{U}_{ad} := \{ \theta \in L^2 (S\times\Omega): \theta_{min} \leq \theta \leq \theta_{max}\,\text{a.e. in }S \times \Omega\}. \nonumber
\end{equation}  In this article, we aim to investigate the following distributed optimal control problem: Find $\theta \in \mathcal{U}_{ad}$ such that 
\begin{align}
    \mathcal{J}(\theta):= \frac{1}{2}\int_{S \times \Omega}|\phi-\phi_d|^2\, dx\,dt+\frac{\delta}{2}\int_{S \times \Omega}|\theta|^2 \, dx\,dt \label{equ 1}
\end{align}
reaches its infimum with respect to $\theta$ subject to both the control constraint 
\begin{equation}
    \theta \in \mathcal{U}_{ad} \label{equ 2}
\end{equation}
and the state equations
\begin{align}
    \partial_t m &= \nabla \cdot [\nabla m - 2 \beta (\phi-m^2)(\nabla J * m)] && \text{in} \,\, S \times \Omega,\label{equ 3}\\
    \partial_t \phi & = \nabla \cdot [\nabla \phi - 2 \beta m(1-\phi)(\nabla J * m)]+ \alpha (1-\phi) + \theta && \text{in} \,\, S \times \Omega,\label{equ 4}\\
    m(t&=0)= m_0 \,\, \text{and} \,\, \phi(t=0)= \phi_0  && \text{in} \,\,  \overline{\Omega},\label{equ 5}\\
    \phi, \,& m \,\,\text{are both periodic} && \text{on} \, \, \partial \Omega, \label{equ 6}
\end{align}
where $t \in S$ and $x \in \Omega$ denote the time and space variable, respectively, $m$ acts as a phase field indicator distinguishing between two polymeric phases, $\beta >0 $ denotes the inverse temperature, $J \in C^2_{+}(\overline{\Omega})$ is symmetric compactly supported potential normalized so that $ \int_{\Omega} J(x) \,dx=1$ and  $\theta$ is the control variable. For the variable $\phi$, $1-\phi$ represents the concentration of a solvent interacting with the polymers. 

It is worth noting that this evolution problem is a parabolic system coupled in terms of a nonlinear, nonlocal, degenerating drift. Its particular structure (with no productions by evaporation) was derived in \cite{Marra} as a suitable hydrodynamic limit of a Blume-Capel-Kac model with Kawasaki dynamics. 

We refer to the state equations \eqref{equ 3}--\eqref{equ 6} by ($\mathcal{P}$) and to the entire control problem \eqref{equ 1}--\eqref{equ 6} by ($\mathcal{OP})$.

\subsection{Organization of this work}
The remainder of the paper is structured as follows. In Section 2, we introduce the necessary mathematical preliminaries, including the definition of weak solutions and a continuous dependence result for the proposed problem. Section 3 is devoted to the main results, where we establish the existence of an optimal solution to the optimal control problem $\mathcal{(OP)}$, prove the Fréchet differentiability of the control-to-state operator, and finally, derive the first-order optimality condition through the associated adjoint problem. Section 4 explores possible extensions of the present study.
\section{Preliminaries}

\subsection{Notations}
For the reader's convenience, we recall some of the more relevant notation here. The notations $L^2(\Omega)$ and $H^1(\Omega)$ conventionally denote the Lebesgue and Sobolev spaces with their respective norms 
\begin{align*}
    ||u||_{L^2(\Omega)}:= \Big( \int_{\Omega}|u(x)|^2\,dx\Big)^\frac{1}{2}
\end{align*}
and 
\begin{align*}
    ||u||_{H^1(\Omega)}:= ||u||_{L^2(\Omega)}+ ||\nabla u||_{L^2(\Omega)}.
\end{align*}
The spaces $C^\infty(\Omega)$ and $C^\infty_{\#} (\Omega)$ denote the space of the smooth functions and the space of all smooth functions which are periodic on $\Omega$. We denote $H^1_{\#}(\Omega)$ as the closure of \textbf{$C^\infty_{\#} (\Omega)$} in the $H^1-$ norm and by $H^{-1}_{\#}(\Omega)$ the corresponding dual space. For convenience of writing, we will often omit to write the index $\#$ in the Sobolev space $H^1_\#(\Omega)$, so we will simply write $H^1(\Omega)$, but we still want the periodicity property to hold.  

We denote $L^p(0,T;B)$ as one of Bochner spaces  with the norm
\begin{align*}
    ||u||_{L^p(0,T;B)}:= \Big( \int_{0}^{T} ||u(t,x)||^p\,dx\Big)^\frac{1}{p}.
\end{align*}
Here, we are going to mostly use $p=2$ and $B$ is one of $L^2(\Omega), \, H^1(\Omega),\, H^1_{\#}(\Omega)$ and $H^{-1}(\Omega)$.  For the reader’s convenience, we introduce the notation:
\begin{equation*}
    \mathcal{W}:= \{ u: u \in L^2(S; H^1(\Omega)),\, \partial_t u \in L^2(S: H^{-1}(\Omega))\}.
\end{equation*}
More details on the definition and properties of the employed function spaces can be found, for instance, in \cite{brezis2011functional, leoni2017first}.
\subsection{Working assumptions}
Our study relies on the following  assumptions:
\begin{itemize}
    \item[(A1)] $\beta>0$.
    \item[(A2)]  $0 \leq |m_0| \leq |\phi_0| \leq 1$ a.e. in $\overline{\Omega}$.
    \item[(A3)] $\mathcal{U}$ is a bounded open subset of $L^2(S \times \Omega)$ containing $\mathcal{U}_{ad}$ and there exists a constant $K >0$ such that 
    \begin{equation*}
        ||\theta||_{L^2(S\times \Omega)} \leq K \,\, \text{for all}\,\, \theta \in \mathcal{U}.
    \end{equation*}
    \item[(A4)] $\delta >0$, $\phi_d \in L^2(S \times \Omega)$, $\theta_{min}, \theta_{max} \in L^\infty(S \times \Omega)$ such that $\theta_{min} \leq \theta_{max}$ for a.e. in $S \times \Omega$. 
\end{itemize}
All these assumptions are both physically and mathematically well motivated. (A1) and (A2) are needed not only to be able to speak about a good concept of solutions (see Section \ref{concept_of_solution}), but they turn to be essential for the phase separation property to hold. In this spirit, Theorem 1.1 cf.  \cite{lyons2024phase} guarantees that 
$$0 \leq |m| \leq |\phi| \leq 1 \mbox{ a.e. in } \overline{\Omega} \times S.$$
It is mathematically interesting to study the case $\beta\to \infty$ (i.e. the limit towards a complete freezing of the dynamics), while  the case $\beta<0$ is mathematically treatable, but physically it is a nonsense. (A3) and (A4) are ingredients needed for the current optimal control formulation.    

\subsection{Concept of  solution and basic properties}\label{concept_of_solution}

\begin{definition}\label{defi 1}
    We say that a pair $(m,\phi) \in \Big( L^2(S;H^1_{\#}(\Omega)) \cap L^\infty (S;L^2(\Omega))\Big)^2$ with $(\partial_t m, \partial_t \phi) \in \Big(L^2(S;H^{-1}_{\#}(\Omega))\Big)^2$ 
is a weak solution of the state equations $(\mathcal{P})$, if it satisfies $m(0,x)=m_0(x)$ and $\phi(0,x)=\phi_0(x)$ for all $x \in \Omega$ together with the identities
\begin{align}
    \langle \partial_t m, \psi\rangle\,dt + \int_{ \Omega} \nabla m \cdot \nabla \psi\, dx  =& \int_{\Omega} 2 \beta (\phi-m^2)(\nabla J * m) \cdot \nabla  \psi\, dx,\label{equ 7}\\    
    \langle \partial_t \phi , \eta\rangle+ \int_{ \Omega} \nabla \phi \cdot \nabla \eta\, dx =& \int_{ \Omega} 2 \beta m (1-\phi)(\nabla J * m) \cdot \nabla  \eta\, dx \nonumber\\
    & + \alpha \int_{ \Omega}(1-\phi)\eta\, dx+  \int_{ \Omega} \theta \eta \, dx,\label{equ 8}
\end{align}
for all $(\psi, \eta ) \in  H^1_{\#}(\Omega) \times H^1_{\#}(\Omega)$.
\end{definition}
Theorem 3.1, and respectively, Lemma 3.3 (both from \cite{lyons2024phase}) ensure the existence, and respectively, the uniqueness of solutions in the sense of Definition \ref{defi 1}. For convenience of the writing in the next sections, we omit to write the index $\#$ in the Sobolev space $H^1_\#(\Omega)$; so, instead of $H^1_\#(\Omega)$, we will simply write $H^1(\Omega)$.  

\begin{theorem}\label{thm 1}
    Let the assumptions (A1), (A2), (A3) be satisfied. Then the solution $(m,\phi)$ in the sense of Definition \ref{defi 1}  satisfies the following {\em a priori} estimate
    \begin{align*}
        || m||_{L^2(S;H^1(\Omega))}+ || \partial_t m||_{L^2(S;H^{-1}(\Omega))} +  || \phi ||_{L^2(S;H^1(\Omega))}+ || \partial_t \phi||_{L^2(S;H^{-1}(\Omega))} \leq C ,
    \end{align*}
    where $C=C(||m_0||_{L^2(\Omega)},\, ||\phi_0||_{L^2(\Omega)},\, || \theta||_{L^2(S \times \Omega)},T)>0$ is a finite constant.
\end{theorem}
\begin{proof}
    Multiplying \eqref{equ 3} and $\eqref{equ 4}$ by $m $ and $\phi$, respectively and integrating the resulting equalities over $\Omega$, we obtain
    \begin{align}
     \frac{1}{2} \frac{d}{dt} || m||^2_{L^2(\Omega)}+ || \nabla m||^2_{L^2(\Omega)} =& \int_{\Omega} 2 \beta (\phi-m^2)(\nabla J * m)\cdot \nabla m \,dx\nonumber\\
     \leq &\, 2 \beta ||\nabla m||_{L^2(\Omega)} || \phi-m^2||_{L^\infty(\Omega)} || \nabla J ||_{L^1(\Omega)} ||m||_{L^2(\Omega)} \label{eq 8}
    \end{align}
    and 
    \begin{align}
        \frac{1}{2} \frac{d}{dt} || \phi||^2_{L^2(\Omega)}+ || \nabla \phi||^2_{L^2(\Omega)} =& 2 \beta \int_{\Omega} m(1-\phi)(\nabla J * m)\cdot \nabla \phi \,dx + \alpha \int_{\Omega} (1-\phi)\phi\, dx+ \int_{\Omega} \theta \phi \, dx \nonumber\\
        \leq &\, 2 \beta || m(1-\phi)||_{L^\infty(\Omega)} || \nabla J||_{L^1} ||m||_{L^2(\Omega)} || \nabla \phi||_{L^2(\Omega)}+ C_1|| \phi||^2_{L^2(\Omega)} \nonumber\\
        &+ ||\theta||_{L^2(\Omega)} || \phi||_{L^2(\Omega)}. \label{eq 9}
    \end{align} 
Now, from Young's and Grönwall's inequalities, \eqref{eq 8} and \eqref{eq 9} imply that 
\begin{equation*}
     ||m||_{L^2(S; H^1(\Omega))}+ ||\phi||_{L^2(S; H^1(\Omega))}\leq C_1,
\end{equation*}
where $C_1=C_1(||m_0||_{L^2(\Omega)},\, ||\phi_0||_{L^2(\Omega)},\, || \theta||_{L^2(S \times \Omega)},T)$.
\par Furthermore, choosing $\psi \in H^1(\Omega)$ yields
\begin{align}
    |\langle\partial_t m , \psi \rangle| \leq & || \nabla m||_{L^2(\Omega)} || \nabla \psi||_{L^2(\Omega)}+ 2 \beta || \phi-m^2||_{L^2(\Omega)} || \nabla J * m ||_{L^\infty(\Omega)} || \nabla \psi ||_{L^2(\Omega)}\nonumber\\
    \leq & \Big\{|| \nabla m||_{L^2(\Omega)} + 2 \beta || \phi-m^2||_{L^2(\Omega)} || \nabla J * m ||_{L^\infty(\Omega)} \Big\} ||\psi ||_{H^1(\Omega)}\nonumber.
    \end{align}
    We then get
    \begin{align}   &\underset{||\psi||_{H^1(\Omega) }\leq 1}{\text{sup}} |\langle \partial_t m, \psi\rangle| \leq  || \nabla m||_{L^2(\Omega)} + 2 \beta || \phi-m^2||_{L^2(\Omega)} || \nabla J * m ||_{L^\infty(\Omega)}
           \end{align}
           and hence, 
         $$|| \partial_t m||_{L^2(S; H^{-1}(\Omega))} \leq C_2.$$
   
    Similarly, one obtains $|| \partial_t \phi||_{L^2(S; H^{-1}(\Omega))} \leq C_3$, where the constants $C_2$ and $C_3$ depend on initial conditions, $K$ and $T$.

\end{proof}

\begin{theorem}\label{thm 2}
    Let $(m_1, \phi_1)$, $(m_2, \phi_2)$ be two solutions in the sense of Definition \ref{defi 1}  for given control functions $\theta_1$, $\theta_2 \in \mathcal{U}$, respectively. Then the following estimate holds:
    \begin{align}
        || m_1- m_2||_{L^2(S;H^1(\Omega))}+ || \partial_t m_1 -\partial_t m_2||_{L^2(S;H^{-1}(\Omega))}+ || \phi_1- \phi_2||_{L^2(S;H^1(\Omega))}\nonumber\\
        +|| \partial_t \phi_1 -\partial_t \phi_2||_{L^2(S;H^{-1}(\Omega))}
        \leq C || \theta_1 - \theta_2||_{L^2(S \times \Omega)},
    \end{align}
    where the constant $C>0$ depends on $K,\,T$ and the initial data.
\end{theorem}
\begin{proof}
Let $(m_1, \phi_1)$ and $(m_2, \phi_2)$ be two solutions of \eqref{equ 3}--\eqref{equ 5} with the same initial conditions and controls $\theta_1$ and $\theta_2$, respectively. Define $m:=m_1-m_2, \, \phi:= \phi_1 - \phi_2$ and $\theta:= \theta_1 - \theta_2$ for $t \in S$. Then the pair $(m, \phi)$ satisfies the equations
\begin{align}
   \partial_t m &= \Delta m - 2 \beta \nabla \cdot [ (\phi_1-m^2_1)(\nabla J * m_1)- (\phi_2-m^2_2)(\nabla J * m_2)] && \text{in} \,\, S \times \Omega,\label{equ 9}\\
    \partial_t \phi & =\Delta \phi - 2 \beta \nabla \cdot [ m_1(1-\phi_1)(\nabla J * m_1)- m_2(1-\phi_2)(\nabla J * m_2)] -\alpha\phi + \theta && \text{in} \,\, S \times \Omega,\label{equ 10}\\
    m(t&=0)=0,\,\, \,\, \phi(t=0)= 0  && \text{in} \,\,  \overline{\Omega},\label{equ 11} \\
    m & =0, \,\,  \,\, \phi= 0  && \text{on} \,\,  S \times \partial \Omega.\label{equ }
\end{align}
Multiplying \eqref{equ 9} by $m$ and then integrating the resulting equality over $\Omega$ yields
\begin{align}
    &\frac{1}{2}||m||^2_{L^2(\Omega)}+ ||\nabla m||^2_{L^2(\Omega)}\nonumber\\
    =& 2 \beta \Big[\int_{\Omega} (\phi_1-m^2_1)(\nabla J * m_1)\cdot \nabla m \, dx- \int_{\Omega}(\phi_2-m^2_2)(\nabla J * m_2) \cdot \nabla m\, dx\Big]\label{equ 12}.
\end{align}
At this point, it is convenient to  observe the following structure in the drift term:
\begin{align}
    &(\phi_1-m^2_1)(\nabla J * m_1)- (\phi_2-m^2_2)(\nabla J * m_2)\nonumber\\
    =\,& \phi_1(\nabla J * m_1)- \phi_2 (\nabla J * m_1)+ \phi_2 (\nabla J * m_1)- \phi_2 (\nabla J * m_2)\nonumber\\
    &\,\,- m^2_1 (\nabla J * m_1) + m^2_2 (\nabla J * m_1)-m^2_2 (\nabla J * m_1)+m^2_2 (\nabla J * m_2)\nonumber\\
    =&\, \phi (\nabla J * m_1) + \phi_2 (\nabla J * m)- m(m_1+m_2)(\nabla J * m_1)-m^2_2 (\nabla J * m)\label{equ 13}.
\end{align}
As a consequence, \eqref{equ 13} turns now \eqref{equ 12} into 
\begin{align}
    &\frac{1}{2}\frac{d}{dt}||m||^2_{L^2(\Omega)}+ ||\nabla m||^2_{L^2(\Omega)}\nonumber\\
    \leq&\, 2 \beta \Big\{||\phi||_{L^2(\Omega)} || \nabla m ||_{L^2(\Omega)} || \nabla J * m_1||_{L^\infty(\Omega)}+ ||\phi_2||_{L^2(\Omega)} ||\nabla m||_{L^2(\Omega)} || \nabla J * m||_{L^\infty(\Omega)}\nonumber\\
    & \, +||m||_{L^4(\Omega)}||\nabla m||_{L^2(\Omega)} || m_1+m_2||_{L^4(\Omega)}||\nabla J * m_1||_{L^\infty(\Omega)}\nonumber\\
    &\, + ||m_2||^2_{L^4(\Omega)} || \nabla J * m||_{L^\infty(\Omega)} || \nabla m ||_{L^2(\Omega)}\Big\}. \label{equ 14}
\end{align}
Now, multiplying \eqref{equ 9} by $\phi$ and integrating the resulting equality over $\Omega$, we obtain
\begin{align}
     &\frac{1}{2}||\phi||^2_{L^2(\Omega)}+ ||\nabla \phi||^2_{L^2(\Omega)}\nonumber\\
     = & 2 \beta \int_{\Omega} [ m_1(1-\phi_1)(\nabla J * m_1)- m_2(1-\phi_2)(\nabla J * m_2)]\cdot \nabla \phi \, dx-  \alpha\int_{\Omega} \phi^2 \, dx + \int_{\Omega}  \theta \phi \, dx. \label{equ 15}
\end{align}
It holds
\begin{align}
    &m_1(1-\phi_1)(\nabla J * m_1)- m_2(1-\phi_2)(\nabla J * m_2)\nonumber\\
    =&\, m_1 (\nabla J * m_1) -m_2 (\nabla J * m_1)+ m_2 (\nabla J * m_1)- m_2 (\nabla J * m_2)+ m_1 \phi_1 (\nabla J * m_1)\nonumber\\
    & - m_2 \phi_1 (\nabla J * m_1)+ m_2 \phi_1 (\nabla J * m_1)-m_2 \phi_2 (\nabla J * m_1)+ m_2 \phi_2 (\nabla J * m_1)-  m_2 \phi_2 (\nabla J * m_2)\nonumber\\
    = & \, m (\nabla J * m_1) + m_2 (\nabla J * m)+ m \phi_1 (\nabla J * m_1)+ m_2 \phi (\nabla J * m_1)+ m_2 \phi_2 (\nabla J * m).
\end{align}
Hence, \eqref{equ 15} becomes
\begin{align}
     &\frac{1}{2}\frac{d}{dt}||\phi||^2_{L^2(\Omega)}+ ||\nabla \phi||^2_{L^2(\Omega)}\nonumber\nonumber\\
     \leq & 2 \beta \Big\{||m||_{L^2(\Omega)} || \nabla \phi||_{L^2(\Omega)} || \nabla J * m_1 ||_{L^\infty(\Omega)}+ ||m_2 ||_{L^2(\Omega)} ||\nabla \phi||_{L^2(\Omega)}|| \nabla J * m||_{L^\infty(\Omega)}\nonumber\\
     & \,+ ||m||_{L^4} || \nabla \phi||_{L^2(\Omega)} || \phi_1||_{L^4(\Omega)} || \nabla J * m_1||_{L^\infty(\Omega)}+||m_2||_{L^4(\Omega)} || \phi||_{L^4(\Omega)} || \nabla \phi||_{L^2(\Omega)} || \nabla J * m_1 ||_{L^\infty(\Omega)}\nonumber\\ 
     &\,+|| m_2||_{L^4(\Omega)} || \phi_2||_{L^4(\Omega)} || \nabla \phi||_{L^2(\Omega)} || \nabla J * m||_{L^\infty(\Omega)}\Big\}+ ||\theta||_{L^2(\Omega)} ||\phi||_{L^2(\Omega)} +  \alpha ||\phi||^2_{L^2(\Omega)}.  \label{equ 17}
\end{align}
As next step in the calculations, we add \eqref{equ 14} and \eqref{equ 17} and use the facts $|| \nabla J * u||_{L^\infty(\Omega)} \leq || \nabla J||_{L^2(\Omega)}|| u||_{L^2(\Omega)}$ and $(m_i, \phi_i) \in \Big( L^2(S;H^1_{\#}(\Omega)) \cap L^\infty (S;L^2(\Omega))\Big)^2$. These ingredients combined with Young's and Grönwall's inequalities imply  that 
\begin{equation*}
    || m||_{L^2(S;H^1(\Omega))}+ || \phi||
    _{L^2(S;H^1(\Omega))} \leq C_1 || \theta||_{L^2(S \times \Omega)},
\end{equation*}
where the constant $C_1 >0 $ depends on $T$.
\par Furthermore, choose $\psi \in H^1(\Omega)$ to get:
\begin{align*}
    |\langle \partial_t m, \psi\rangle| \leq&  \, || \nabla m||_{L^2(\Omega)}|| \nabla \psi||_{L^2(\Omega)}+ 2 \beta \Big\{ || \phi||_{L^2(\Omega)} ||\nabla J * m_1||_{L^\infty(\Omega)} +||\phi_2||_{L^2(\Omega)}|| \nabla J * m||_{L^\infty(\Omega)} \\
&+||m||_{L^4(\Omega)}||m_1+m_2||_{L^4(\Omega)}   ||\nabla J * m_1||_{L^\infty(\Omega)} +||m_2||^2_{L^4(\Omega)} || \nabla J * m||_{L^\infty(\Omega)}\}|| \nabla \psi||_{L^2(\Omega)}\\
    \leq & \Big\{  || \nabla m||_{L^2(\Omega)} + 2 \beta \Big( || \phi||_{L^2(\Omega)} ||\nabla J * m_1||_{L^\infty(\Omega)} +||\phi_2||_{L^2(\Omega)}|| \nabla J * m||_{L^\infty(\Omega)} \\
&+||m||_{L^4(\Omega)}||m_1+m_2||_{L^4(\Omega)}   ||\nabla J * m_1||_{L^\infty(\Omega)} +||m_2||^2_{L^4(\Omega)} || \nabla J * m||_{L^\infty(\Omega)}\Big) \Big\} ||\psi||_{H^1(\Omega)}.
\end{align*}
Hence, it holds 
\begin{align*}
    \underset{||\psi||_{H^1(\Omega) }\leq 1}{\text{sup}} |\langle \partial_t m, \psi\rangle| \leq &  || \nabla m||_{L^2(\Omega)} + 2 \beta \Big( || \phi||_{L^2(\Omega)} ||\nabla J * m_1||_{L^\infty(\Omega)} +||\phi_2||_{L^2(\Omega)}|| \nabla J * m||_{L^\infty(\Omega)} \\
    &+||m||_{L^4(\Omega)}||m_1+m_2||_{L^4(\Omega)}   ||\nabla J * m_1||_{L^\infty(\Omega)} +||m_2||^2_{L^4(\Omega)} || \nabla J * m||_{L^\infty(\Omega)}\Big),
    \end{align*}
    which finally ensures 
    $$|| \partial_t m||_{L^2(S;H^{-1}(\Omega))} \leq \, C_2 ||\theta ||_{L^2(S \times \Omega)}.$$
Similarly, one obtains
\begin{align*}
     \underset{||\psi||_{H^1(\Omega) }\leq 1}{\text{sup}} |\langle \partial_t \phi, \eta\rangle| \leq & || \nabla \phi||_{L^2(\Omega)} + 2 \beta\Big\{||m||_{L^2(\Omega)} || \nabla J * m_1 ||_{L^\infty(\Omega)}+ ||m_2 ||_{L^2(\Omega)}|| \nabla J * m||_{L^\infty(\Omega)}\nonumber\\
     & \,+ ||m||_{L^4}  || \phi_1||_{L^4(\Omega)} || \nabla J * m_1||_{L^\infty(\Omega)}+||m_2||_{L^4(\Omega)} || \phi||_{L^4(\Omega)} || \nabla J * m_1 ||_{L^\infty(\Omega)}\nonumber\\ 
     &\,+|| m_2||_{L^4(\Omega)} || \phi_2||_{L^4(\Omega)} || \nabla J * m||_{L^\infty(\Omega)}\Big\}+ ||\theta||_{L^2(\Omega)}  +  \alpha ||\phi||_{L^2(\Omega)}\\
      \Rightarrow \,  || \partial_t \phi||_{L^2(S;H^{-1}(\Omega))} \leq &\, C_3 ||\theta ||_{L^2(S \times \Omega)}. 
\end{align*}
Putting all these facts together, we simply conclude that the control-to-state-operator i.e., $\mathcal{S}: \theta \to (m, \phi)$ is Lipschitz continuous from $L^2(S \times \Omega)$ to $$\mathcal{F}:= H^1(S; H^{-1}(\Omega)) \cap L^2(S ; H^1(\Omega)) \times H^1(S; H^{-1}(\Omega)) \cap L^2(S ; H^1(\Omega)).$$ 
Hence, the control to state mapping $\mathcal{S}: \theta \to \mathcal{S}(\theta)= (m, \phi)$ is well-defined as a mapping from $L^2(S \times \Omega)$ onto $\mathcal{F}$.  Moreover, $\mathcal{S}$ is a Lipschitz continuous mapping from $\mathcal{U}$ of $L^2(S\times \Omega)$ to $\mathcal{F}$.

\end{proof}

\section{Optimality condition}
\subsection{Existence of an optimal control}
\begin{theorem}
     Assume (A1)--(A2) are satisfied. Then the optimal control problem $\mathcal{(OP)}$ admits a solution.
\end{theorem}
\begin{proof}
    Abusing a bit of notation, we let $m := \underset{\theta \in \mathcal{U}_{\text{ad}}}{\text{inf}}\,\mathcal{J}(\theta)$. Since $0 \leq m < \infty$, there exists a sequence $(\theta_n)_{n} \in \mathcal{U}_{\text{ad}} \subset L^2(S \times \Omega)$ such that $\underset{n \to \infty}{\lim} \mathcal{J}(\theta_n)=m$. Further assume that the pair  $(m_n, \phi_n)$ satisfies the state equations $\mathcal{(P)}$ for a given $\theta_n$ with $n \in \mathbb{N}$. Hence, relying on the {\em a priori} estimates stated in  Theorem \ref{thm 1}, one can extract subsequences $(m_n)_n,\, (\partial_t m_n)_n,\, (\phi_n)_n,\, (\partial_t \phi_n)_n$ and $(\theta_n)_n$, still indexed by $n$ such that the following convergences hold:

\begin{enumerate}
        \item[(i)] $\theta_n  \overset{w}{\rightharpoonup} \theta$ in $\mathcal{U}_{\text{ad}} \subset L^2(S \times \Omega)$,
        \item[(ii)] $m_n\overset{w}{\rightharpoonup} m$ in $L^2(S ; H^1(\Omega))$,
        \item[(iii)] $\partial_t m_n \overset{w}{\rightharpoonup} \partial_t m$ in $L^2(S ; H^{-1}(\Omega))$,
         \item[(iv)] $\phi_n\overset{w}{\rightharpoonup} \phi$ in $L^2(S ; H^1(\Omega))$,
         \item[(v)] $\partial_t \phi_n \overset{w}{\rightharpoonup} \partial_t \phi$ in $L^2(S ; H^{-1}(\Omega))$,
         \item[(vi)] $m_n \to  m$ strongly in $L^2(S \times \Omega)$,
         \item[(vii)] $\phi_n \to  \phi$ strongly in $L^2(S \times \Omega)$.
    \end{enumerate}
    Now, we recall that $J \in C^2_{+}(\overline{\Omega})$ and compactly supported implies that $\nabla J \in L^1(\Omega)$. Then $|| \nabla J * m||_{L^2(\Omega)} \leq C || m||_{L^2(\Omega)}$. So the map $ m \to \nabla J * m$ is linear and continuous from $L^2(\Omega)  \to L^2(\Omega)$. Hence,
    \begin{equation*}
        m_n \to m \,\, \text{strongly in} \, \, L^2(\Omega) \,\text{implies} \,\nabla J * m_n  \to \nabla J * m \,\, \text{strongly in} \, \, L^2(\Omega).
    \end{equation*}
    Hence, $(m, \phi)$ is the weak solution of the state equations corresponding the to the control $\theta$. Next, the weakly lower semicontinuity of $\mathcal{J}$ implies that $\theta \in \mathcal{U}_{ad}$ is an optimal control for the proposed optimal control problem; see e.g. \cite{jpraymond, troltzsch2010optimal} for more background in the context of optimal control problems posed  for PDEs. 
\end{proof}
Within this framework, the symbol $\hat{\theta} \in \mathcal{U}_{ad}$ consistently denotes a locally optimal control, while $(\hat{m}, \hat{\phi}) = \mathcal{S}(\hat{\theta})$ represents its corresponding state.

\subsection{Structure of the linearized system} Let $h \in L^2(S \times \Omega)$ be fixed arbitrarily. Now, to be able to ensure the Fréchet differentiability of the control to state operator $\mathcal{S}$, we need to find firstly the linearized form of the state equations $(\mathcal{P})$ around the solution pair $(\hat{m}, \hat{\phi})$. Taking $m= \hat{m} + \varphi_1, \, \phi= \hat{\phi} + \varphi_2$ and substituting this in $\mathcal{(P)}$, we obtain the following linearized system

\begin{align}
    \partial_t \varphi_1 =& \nabla \cdot [ \nabla \varphi_1- 2 \beta (\varphi_2 - 2 \hat{m}\varphi_1)(\nabla J * \hat{m})-2 \beta (\hat{\phi}-\hat{m}^2)(\nabla J * \varphi_1)] \,\, \text{in}\,\, S \times \Omega,\label{lin 1}\\
    \partial_t \varphi_2=& \nabla \cdot[ \nabla \varphi_2- 2\beta \hat{m}(1- \hat{\phi})(\nabla J * \varphi_1)-2 \beta \varphi_1 (1- \hat{\phi})(\nabla J * \hat{m})\nonumber\\
    & \,\,+ 2 \beta \varphi_2 \hat{m}(\nabla J * \hat{m})]- \alpha \varphi_2 +h  \,\,\hspace{3cm}\text{in}\,\, S \times \Omega,\label{lin 2}\\
   \varphi_1, \varphi_2 &\text{ are periodic}\hspace{5.9cm}\,\, \text{on}\,\, \partial \Omega,\\
    \varphi_1(t&=0)= 0, \,\,  \,\, \varphi_2(t=0)= 0  \hspace{4.1cm} \text{in} \,\,  \overline{\Omega}\label{lin 3}.
\end{align}
\begin{theorem}\label{thm 4}
    For any $h \in L^2(S \times \Omega)$, the linearized system defined by \eqref{lin 1}--\eqref{lin 3} possesses a unique weak solution $$(\varphi_1, \varphi_2) \in  \mathcal{W} \times \mathcal{W}.$$ Furthermore, the following stability estimate holds
    \begin{align*}
        || \varphi_1||_{L^2(S; H^1(\Omega))} +|| \partial_t \varphi_1 ||_{L^2(S; H^{-1}(\Omega))} + || \varphi_2||_{L^2(S; H^1(\Omega))} +|| \partial_t \varphi_2 ||_{L^2(S; H^{-1}(\Omega))} \leq C|| h||_{L^2(S \times \Omega)},
    \end{align*}
    where $C > 0$ has the same dependency as in the statement of Theorem \ref{thm 1}.  
    \end{theorem}
    \begin{proof}
     The existence of weak solutions to the linearized system \eqref{lin 1}---\eqref{lin 2} follows from a straightforward application of the Galerkin approximation method, see e.g. \cite{evans2022partial}. 
     
     In the context of this proof, we are  mainly concerned with ensuring the stated stability bound with respect to $h \in L^2(S \times \Omega)$. Multiplying \eqref{lin 1} by $\varphi_1$ and then integrating the resulting equalities over $\Omega$, we obtain the following expressions:
        \begin{align}
            &\frac{1}{2} \frac{d}{dt} ||\varphi_1||^2_{L^2(\Omega)} + || \nabla \varphi_1||^2_{L^2(\Omega)}\nonumber\\
            = & 2 \beta \int_{\Omega} \Big\{(\varphi_2 - 2 \hat{m}\varphi_1)(\nabla J * \hat{m})+ (\hat{\phi}-\hat{m}^2)(\nabla J * \varphi_1)\Big\} \cdot \nabla \varphi_1\nonumber\\
            \leq & 2\beta \int_{\Omega} \Big(|| \varphi_2 - 2 \hat{m}\varphi_1||_{L^2(\Omega)} || \nabla J * \hat{m}||_{L^\infty(\Omega)}+ || \hat{\phi}-\hat{m}^2||_{L^2(\Omega)} || \nabla J * \varphi_1||_{L^\infty(\Omega)}\Big) || \nabla \varphi_1||_{L^2( \Omega)}\nonumber\\
            \leq&  2 \beta || \varphi_2 ||_{L^2(\Omega)} ||\nabla J * \hat{m}||_{L^\infty(\Omega)}  || \nabla \varphi_1||_{L^2(\Omega)} + 4 \beta || \hat{m}||_{L^\infty(\Omega)} ||\nabla J * \hat{m}||_{L^\infty(\Omega)} || \varphi_1 ||_{L^2(\Omega)}|| \nabla \varphi_1||_{L^2(\Omega)}\nonumber\\
            & +2 \beta|| \hat{\phi}-\hat{m}^2||_{L^2(\Omega)} || \nabla J ||_{L^2(\Omega)}|| \varphi_1 ||_{L^2(\Omega)}|| \nabla \varphi_1||_{L^2(\Omega)}. \label{lin 4}
        \end{align}
Again, multiplying \eqref{lin 2} by $\varphi_2$ and integrating the resulting equalities over $\Omega$, we obtain
\begin{align}
     &\frac{1}{2} \frac{d}{dt} ||\varphi_2||^2_{L^2(\Omega)} + || \nabla \varphi_2||^2_{L^2(\Omega)}\nonumber\\
     = &  \int_{\Omega}2 \beta  \Big\{  \hat{m}(1- \hat{\phi})(\nabla J * \varphi_1)+ \varphi_1 (1- \hat{\phi})(\nabla J * \hat{m})- \varphi_2 \hat{m}(\nabla J * \hat{m})\Big\}\cdot \nabla \varphi_2\, dx + \int_{ \Omega} (h - \alpha \varphi_2)\varphi_2\, dx \nonumber\\
     \leq & 2\beta \Big\{|| \hat{m}(1- \hat{\phi})||_{L^2(\Omega)} || \nabla J||_{L^2(\Omega)} || \varphi_1||_{L^2(\Omega)}+ || \varphi_1||_{L^2(\Omega)} ||(1- \hat{\phi})(\nabla J * \hat{m})||_{L^\infty(\Omega)} \nonumber\\
     &+ || \varphi_2||_{L^2(\Omega)} ||\hat{m}(\nabla J * \hat{m})||_{L^\infty(\Omega)}\Big\} || \nabla \varphi_2 ||_{L^2(\Omega)}+ \alpha || \varphi_2 ||^2_{L^2(\Omega)} + || h||_{L^2( \Omega)} || \varphi_2||_{L^2(\Omega)}.\label{lin 5}
\end{align}
Now, adding \eqref{lin 4} and \eqref{lin 5} and applying the Young's inequality, we obtain
\begin{align*}
     \frac{d}{dt} (||\varphi_1||^2_{L^2(\Omega)}+   ||\varphi_2||^2_{L^2(\Omega)})+ (|| \nabla \varphi_1||^2_{L^2(\Omega)}+ || \nabla \varphi_2||^2_{L^2(\Omega)})\\
     \leq C_1 (||\varphi_1||^2_{L^2(\Omega)}+   ||\varphi_1||^2_{L^2(\Omega)}) + C_2 ||h||^2_{L^2(\Omega)},
\end{align*}
where
\begin{align*}
    C_1= C\Big(1+ ||\nabla J * \hat{m}||^2_{L^\infty(\Omega)} + || \hat{m}||^2_{L^\infty(\Omega)} ||\nabla J * \hat{m}||^2_{L^\infty(\Omega)}
    + || \hat{\phi}-\hat{m}^2||^2_{L^2(\Omega)} || \nabla J ||^2_{L^2(\Omega)}\\
    + || \hat{m}(1- \hat{\phi})||^2_{L^2(\Omega)} || \nabla J||^2_{L^2(\Omega)}
    + ||(1- \hat{\phi})(\nabla J * \hat{m})||^2_{L^\infty(\Omega)} + ||\hat{m}(\nabla J * \hat{m})||^2_{L^\infty(\Omega)}\Big).
\end{align*}
Integrating the above inequality over $(0,t)$ and applying the Grönwall's inequality, we obtain
\begin{align*}
    (||\varphi_1(t)||^2_{L^2(\Omega)}+   ||\varphi_2(t)||^2_{L^2(\Omega)})+ \int_{0}^{T} (|| \nabla \varphi_1||^2_{L^2(\Omega)}+ || \nabla \varphi_2||^2_{L^2(\Omega)})\,dt \leq C_2||h||^2_{L^2(S \times \Omega)},
\end{align*}
for a.e. $t \in S$. Furthermore, 
\begin{align*}
    \int_{0}^{T} || \partial_t \varphi_1 ||_{H^{-1}(\Omega)} dt\leq& \int_{0}^{T} \Big\{|| \nabla \gamma_1 ||_{L^2(\Omega)}+|| \varphi_2 - 2 \hat{m}\varphi_1||_{L^2(\Omega)} || \nabla J * \hat{m}||_{L^\infty(\Omega)}\\
    &+ || \hat{\phi}-\hat{m}^2||_{L^2(\Omega)} || \nabla J * \varphi_1||_{L^\infty(\Omega)}\Big\}dt
\end{align*}
and 
\begin{align*}
     \int_{0}^{T} || \partial_t \varphi_2 ||_{H^{-1}(\Omega)} \,dt\leq& \int_{0}^{T} \Big[|| \nabla \gamma_2 ||_{L^2(\Omega)}+ 2 \beta \Big\{|| \hat{m}(1- \hat{\phi})||_{L^2(\Omega)} || \nabla J||_{L^2(\Omega)} || \varphi_1||_{L^2(\Omega)}\nonumber\\
     &+ || \varphi_1||_{L^2(\Omega)} ||(1- \hat{\phi})(\nabla J * \hat{m})||_{L^\infty(\Omega)} + || \varphi_2||_{L^2(\Omega)} ||\hat{m}(\nabla J * \hat{m})||_{L^\infty(\Omega)}\Big\}\\
     &+ \alpha || \varphi_2||_{L^2(\Omega)}+ ||h||_{L^2(\Omega)}\Big]\,dt.
\end{align*}
  
    \end{proof}

\subsection{Differentiability of the control to state operator}

\begin{theorem}
   Let the assumptions (A1)--(A3) hold true. Then, for any $\hat{\theta} \in \mathcal{U}$, the control to state mapping $\mathcal{S}$ is Fréchet differentiable from $L^2(S \times \Omega)$ into the space $\mathcal{F}$. Moreover, for any $h \in L^2(S \times \Omega)$, its Fréchet derivative $\mathcal{D}S$ is given by
    \begin{equation*}
        D \mathcal{S}(\hat{\theta})(h) = (\varphi_1, \varphi_2),
    \end{equation*}
    where the pair $(\varphi_1, \varphi_2)$ is the weak solution to the linearized system \eqref{lin 1}--\eqref{lin 3} with respect to $h $.
\end{theorem}
\begin{proof}

Assume that for any fixed $\hat{\theta} \in \mathcal{U}, \mathcal{S}(\hat{\theta})= (\hat{m}, \hat{\phi})$ be the associated solution to the state equations $\mathcal{(P)}$. Since $\mathcal{U}$ is an open subset of $L^2(S \times \Omega)$, there exists $\lambda > 0$ such that for any $h \in L^2( S\times \Omega)$ with $||h||_{L^2(S \times \Omega)} \leq \lambda $, one can say $ \hat{\theta} + h \in \mathcal{U}$. For such $h \in L^2( S\times \Omega)$, we let $(m^h, \phi^h)$ be the solution to the state equations $\mathcal{(P)}$ with respect to $\hat{\theta}+ h$. Define $m^h-\hat{m}=: y_1$ and $\phi^h-\hat{\phi} =: y_2$, then the pair $(y_1, y_2)$ satisfies 
    \begin{align}
        \partial_t y_1 = & \Delta y_1 - 2 \beta \nabla \cdot [(\phi^h -(m^h)^2)(\nabla J * y_1)+ \{y_2  - y_1(m^h + y_1)\}(\nabla J * \hat{m})] && \text{in}\,\, S \times \Omega,\label{fre 1}\\
        \partial_t y_2 = & \Delta y_2 - 2 \beta \nabla \cdot [ m^h (1-\phi^h)(\nabla J * y_1)+ (y_1-y_1\phi^h - \hat{m}y_2)(\nabla J * \hat{m})]-\alpha y_2+h&& \text{in}\,\, S \times \Omega,\label{fre 2}\\
       y_1 & =0, \,\,  \,\, y_2= 0  && \text{on} \,\,  S \times \partial \Omega,\\
        y_1(0,x&)=0, \,\,y_2(0,x)=0&& \text{in}\,\,  \overline{\Omega}. \label{fre 3}
    \end{align}
\end{proof}
As next step, we define $z_1:= y_1- \varphi_1$ and $z_2:= y_2- \varphi_2$, then $(z_1, z_2)$ satisfies 
\begin{align}
    \partial_t z_1=& \Delta z_1 - 2\beta \nabla \cdot [ (\phi^h- (m^h)^2)(\nabla J * z_1)+ (y_2 -y_1(m^h+\hat{m}))(\nabla J * \varphi_1)\nonumber\\
    & \qquad+ (z_2-2\hat{m}z_1-y^2_1)(\nabla J * \hat{m})] \,\, \text{in}\,\, S \times \Omega\label{fre 4}\\
    \partial_t z_2=& \Delta z_2 - 2\beta \nabla \cdot [m^h(1-\phi^h)(\nabla J * z_1)+ (y_1-m^hy_2+ \hat{\phi} y_1)(\nabla J * \varphi_1)\nonumber\\
    &\qquad+ (z_1 - \phi^h z_1-\varphi_1 y_2-\hat{m}z_2 )(\nabla J * \hat{m})] \,\, \text{in}\,\, S \times \Omega\label{fre 5}\\
    z_1(0,x&)=0,\,\, z_2(0,x)=0\,\, \text{in}\,\,  \overline{\Omega}.\label{fre 6}
\end{align}

Multiplying the equation \eqref{fre 4} by $z_1$ and integrating the resulting equalities over $\Omega$, we obtain
\begin{align}
    &\frac{1}{2} \frac{d}{dt}||z_1||^2_{L^2(\Omega)} + || \nabla z_1||^2_{L^2(\Omega)}\nonumber\\
    =& 2 \beta\int_{\Omega}\Big \{(\phi^h- (m^h)^2)(\nabla J * z_1)+ (y_2 -y_1(m^h+\hat{m}))(\nabla J * \varphi_1)+ (z_2-2\hat{m}z_1-y^2_1)(\nabla J * \hat{m})\Big\}\cdot \nabla z_1\, dx\nonumber\\
    \leq &\, 2 \beta\Big\{|| \phi^h- (m^h)^2||_{L^2(\Omega)}||\nabla J * z_1||_{L^\infty(\Omega)} +||  y_2 -y_1(m^h+\hat{m})||_{L^2(\Omega)}|| \nabla J * \varphi_1||_{L^\infty(\Omega)}\nonumber\\
    &\quad \quad + || z_2-2\hat{m}z_1-y^2_1||_{L^2(\Omega)}|| \nabla J * \hat{m}||_{L^\infty(\Omega)}\Big\} || \nabla z_1||_{L^2(\Omega)}\nonumber\\
    \leq &2 \beta\Big\{|| \phi^h- (m^h)^2||_{L^2(\Omega)} || \nabla J||_{L^2(\Omega)}||z_1||_{L^2(\Omega)} + ||  y_2 -y_1(m^h+\hat{m})||_{L^2(\Omega)}|| \nabla J * \varphi_1||_{L^\infty(\Omega)}\nonumber\\
    &\quad \quad + || z_2-2\hat{m}z_1-y^2_1||_{L^2(\Omega)}|| \nabla J * \hat{m}||_{L^\infty(\Omega)}\Big\} || \nabla z_1||_{L^2(\Omega)} \label{fre 7}
\end{align}
Again, multiplying $z_2$ with \eqref{fre 4} and integrating the resulting equality over $\Omega$, we get 
\begin{align}
    &\frac{1}{2} \frac{d}{dt}||z_2||^2_{L^2(\Omega)} + || \nabla z_2||^2_{L^2(\Omega)}\nonumber\\
    \leq & 2 \beta\int_{\Omega} \Big\{m^h(1-\phi^h)(\nabla J * z_1)+ (y_1-m^hy_2+ \hat{\phi} y_1)(\nabla J * \varphi_1)\nonumber\\
    &\qquad+ (z_1 - \phi^h z_1-\varphi_1 y_2 -\hat{m}z_2)(\nabla J * \hat{m})\Big\} \cdot  \nabla z_2 + \alpha ||z_2||^2_{L^2(\Omega)}\nonumber\\
    & \leq 2 \beta\Big\{|| m^h||_{L^4(\Omega)} ||1-\phi^h||_{L^4(\Omega)} || \nabla J * z_1||_{L^\infty}+ || y_1-m^h y_2 + \hat{\phi} y_1||_{L^2(\Omega)}  || \nabla J * \varphi_1||_{L^\infty(\Omega)}\nonumber\\
    & \,\,+ || z_1 - \phi^h z_1-\varphi_1 y_2 -\hat{m}z_2||_{L^2(\Omega)} || \nabla J * \hat{m}||_{L^\infty(\Omega)} \Big\}|| \nabla z_2||_{L^2(\Omega)}+\alpha || z_2||^2_{L^2(\Omega)} \label{fre 8}
\end{align}
Adding \eqref{fre 7}, \eqref{fre 8} and using the Young's inequality, we obtain
\begin{align}
    \frac{d}{dt}(||z_1||^2_{L^2(\Omega)}+||z_2||^2_{L^2(\Omega)}) + (|| \nabla z_1||^2_{L^2(\Omega)}+ || \nabla z_2||^2_{L^2(\Omega)}) \leq C_1 (||z_1||^2_{L^2(\Omega)}+||z_2||^2_{L^2(\Omega)}) + C_2, \label{fre 9}
\end{align}
where 
\begin{align*}
    C_1= C\Big(1+ ||\phi^h-m^h||^2_{L^2(\Omega)}||\nabla J ||^2_{L^2(\Omega)}+ ||\hat{m}||^2_{L^\infty(\Omega)}|| \nabla J * \hat{m}||^2_{L^\infty(\Omega)}+ || \phi^h||^2_{L^\infty(\Omega)} ||\nabla J * \hat{m}||^2_{L^\infty(\Omega)}\\
    + ||m^h||^2_{L^4(\Omega)}||1-\phi^h||^2_{L^4(\Omega)}|| \nabla J ||^2_{L^2(\Omega)}+ ||\hat{m}||^2_{L^\infty(\Omega)}|| \nabla J * \hat{m}||^2_{L^\infty(\Omega)} \Big)
\end{align*}
and
\begin{align*}
    C_2 = C\Big( ||  y_2 -y_1(m^h+\hat{m})||^2_{L^2(\Omega)}|| \varphi_1||^2_{L^2(\Omega)} || \nabla J ||^2_{L^2(\Omega)}
    + ||y_1||^4_{L^4(\Omega)}||\nabla J * \hat{m}||^2_{L^\infty(\Omega)}\\
    + ||y_1-m^h y_2 + \hat{\phi}y_1||^2_{L^2(\Omega)}||\varphi_1||^2_{L^2(\Omega)} || \nabla J ||^2_{L^2(\Omega)}
    + || \varphi_1||^4_{L^4(\Omega)} || \nabla J * \hat{m}||^2_{L^\infty(\Omega)}\\
    +|| y_2||^4_{L^4(\Omega)}|| \nabla J * \hat{m}||^2_{L^\infty(\Omega)} \Big).
\end{align*}
Now, Gronwall's inequality implies that 
\begin{align}
    ||z_1(t)||^2_{L^2(\Omega)}+||z_2(t)||^2_{L^2(\Omega)}+ \int_{0}^{T} || \nabla z_1||^2_{L^2(\Omega)}+ || \nabla z_2||^2_{L^2(\Omega)})\,dt \lesssim \Big(  \int_{0}^{T}C_2(t) \,dt\Big) e^{\int_{0}^{T}C_1(t)\,dt}.
\end{align}

    Furthermore,
    \begin{align*}
        \int_{0}^{T} || \partial_t z_1||_{H^{-1}(\Omega)}\, dt  \leq &\int_{0}^{T} \Big[ || \nabla z_1||_{L^2(\Omega)}+ 2 \beta \Big\{|| \phi^h- (m^h)^2||_{L^2(\Omega)}||\nabla J||_{L^2(\Omega)}  ||z_1||_{L^2(\Omega)} \nonumber\\
    &\quad \quad +||  y_2 -y_1(m^h+\hat{m})||_{L^2(\Omega)}|| \nabla J * \varphi_1||_{L^\infty(\Omega)}\nonumber\\
    &\quad \quad +|| z_2-2\hat{m}z_1-y^2_1||_{L^2(\Omega)}|| \nabla J * \hat{m}||_{L^\infty(\Omega)}\Big\} \Big]dt
    \end{align*}
and
\begin{align*}
    \int_{0}^{T} || \partial_t z_1||_{H^{-1}(\Omega)}\, dt  \leq 
    & \int_{0}^{T} \Big[ || \nabla z_2||_{L^2(\Omega)}+ 2 \beta \Big\{ || m^h||_{L^4(\Omega)} ||1-\phi^h||_{L^4(\Omega)} || \nabla J * z_1||_{L^\infty}\\
    &+ || y_1-m^h y_2 + \hat{\phi} y_1||_{L^2(\Omega)}  || \nabla J * \varphi_1||_{L^\infty(\Omega)}\nonumber\\
    & \,\,+ || z_1 - \phi^h z_1-\varphi_1 y_2 -\hat{m}z_2||_{L^2(\Omega)} || \nabla J * \hat{m}||_{L^\infty(\Omega)} \Big\}+ \alpha ||z_2||_{L^2(\Omega)}\Big]\,dt.
\end{align*}
Now, from the Theorem \ref{thm 2}, Theorem \ref{thm 4} and the explicit form of the $C_1$ and $C_2$ imply that 
\begin{align}
    ||z_1(t)||^2_{L^2(\Omega)}+||z_2(t)||^2_{L^2(\Omega)}+ \int_{0}^{T} (|| \nabla z_1||^2_{L^2(\Omega)}+ || \nabla z_2||^2_{L^2(\Omega)}+ || \partial_t z_1||_{H^{-1}(\Omega)}+ || \partial_t z_2||_{H^{-1}(\Omega)})\,dt \nonumber\\
    \leq C(t, ||h||_{L^2(S\times \Omega)})||h||^2_{L^2(S\times \Omega)},
\end{align}
where $ \underset{||h||_{L^2(S\times \Omega)} \to 0}{ \lim} C(t, ||h||_{L^2(S\times \Omega)})=0.$ Hence,
\begin{align*}
    \underset{||h||_{L^2(S\times \Omega)} \to 0}{ \lim} \frac{||\mathcal{S}(\hat{\theta}+ h)- \mathcal{S}(\hat{\theta})- \mathcal{D} \mathcal{S}(\hat{\theta}) h||}{||h ||_{L^2(S \times \Omega)}} = 0.
\end{align*}

Next, we derive the necessary variational inequality characterizing optimal controls. Due to the quadratic nature of the cost functional and the application of the chain rule, the reduced functional defined by the composition
\begin{equation*}
    \mathcal{J}(\theta)= \mathcal{J}(\mathcal{S}(\theta), \theta)
\end{equation*}
admits a Fréchet derivative at every $\hat{\theta} \in \mathcal{U}$, where the Fréchet derivative is given by
\begin{align*}
     \mathcal{D} \mathcal{J}(\hat{\theta}) = \mathcal{D}_{(m, \phi)} \mathcal{J}(\mathcal{S}(\hat{\theta}), \hat{\theta})\cdot \mathcal{D} \mathcal{S}( \hat{\theta})+ \mathcal{D}_\theta  \mathcal{J}\mathcal{S}(\hat{\theta}).
 \end{align*}
 Now, using the convexity of $\mathcal{U}_{ad}$ we obtain that for any local minimizer $\hat{\theta} \in \mathcal{U}_{ad}$ of $\mathcal{J}$ in $\mathcal{U}_{ad}$,
 
\begin{equation*}
    \mathcal{D} \mathcal{J}(\hat{\theta})( \theta - \hat{\theta}) \geq 0 \,\, \text{for all}\,\, \theta \in \mathcal{U}_{ad}.
\end{equation*}
We summarize these calculations in the next result.
\begin{theorem}
   Assume that (A1)--(A3) are satisfied and let $ \hat{\theta}$ be an optimal control for the $\mathcal{(OP)}$ with the associated state $\mathcal{S}(\hat{\theta})= ( \hat{m}, \hat{\phi})$, then for $\theta \in L^2(S\times \Omega)$ it holds
    \begin{equation*}
        \int_{S\times \Omega} (\hat{\phi}-\phi_d)\varphi_2 \, dx\,dt+ \delta\int_{S \times \Omega} (\theta- \hat{\theta})\hat{\theta} \,dx\,dt \geq 0,
    \end{equation*}
    where $\mathcal{D} \mathcal{J}(\hat{\theta})( \theta - \hat{\theta})= (\varphi_1, \varphi_2)$ is the unique weak solution to the linearized system \eqref{lin 1}--\eqref{lin 3} for $h:= \theta - \hat{\theta}$.
    \end{theorem}

\subsection{Adjoint problem}
We start off by constructing  the following Lagrangian function 
\begin{equation*}
    \mathfrak{J}: \mathcal{W}\times \mathcal{W} \times \mathcal{U}_{ad} \times \mathcal{W}\times \mathcal{W} \to \mathbb{R}
\end{equation*}
such that 
\begin{align*}
    \mathfrak{J}(m, \phi, \theta, \gamma_1, \gamma_2):= \frac{1}{2}\int_{S \times \Omega}|\phi-\phi_d|^2\, dx\,dt+\frac{\delta}{2}\int_{S \times \Omega}|\theta|^2 \, dx\,dt\nonumber\\
    -\int_{S \times \Omega} \Big[\partial_t m - \nabla \cdot \{\nabla m - 2 \beta (\phi-m^2)(\nabla J * m)\}\Big]\gamma_1\,dx dt\nonumber\\
    -\int_{S\times \Omega}\Big[\partial_t \phi - \nabla \cdot \{\nabla \phi - 2 \beta m(1-\phi)(\nabla J * m)\}+ \alpha (1-\phi) + \theta\Big] \gamma_2\,dx\,dt.
\end{align*}
Step 1: We compute the Gâteaux derivative of the Lagrangian with respect to the variable $m$. The Gâteaux derivative of $\mathfrak{J}$ taken at $(\hat{m},\hat{\phi},\hat{\theta})$ in the direction $h_1$ is defined by

\begin{align*}
    &\mathcal{D}_m\mathfrak{J}(\hat{m}, \hat{\phi}, \hat{\theta})h_1= 0\, \, \text{in other words} \,\, \frac{d}{d \lambda} \mathfrak{J}(\hat{m}+ \lambda h_1)\Big|_{\lambda = 0}=0\,\, \text{ for all}\,\, h_1 \,\, \text{with $h_1(0)=0$}.
\end{align*}
Now,
 \begin{align*}
     &\frac{d}{d \lambda} \mathfrak{J}(\hat{m}+ \lambda h_1)\\
 = &- \frac{d}{d \lambda} \int_{S \times \Omega} \Big[ \partial_t(\hat{m}+ \lambda h_1)- \nabla \cdot \big\{\nabla(\hat{m}+ \lambda h_1)- 2 \beta (\phi- (\hat{m}+ \lambda h_1)^2) (\nabla J *(\hat{m}+ \lambda h_1))\}\Big]\gamma_1\, dx\,dt\\
 & \,-\frac{d}{d \lambda} \int_{S \times \Omega} \Big[\partial_t \phi - \nabla \cdot \{ \nabla \phi - 2 \beta(\hat{m}+ \lambda h_1)(1-\phi)(\nabla J * (\hat{m} + \lambda h_1))\}+ \alpha(1-\phi) + \theta \Big]\gamma_2 \,dx\,dt\\
 =& -\int_{S \times \Omega} \Big[ \partial_t h_1- \nabla \cdot \{ \nabla h_1+4\beta (\hat{m}+ \lambda h_1)h_1 (\nabla J * (\hat{m} + \lambda h_1))- 2 \beta (\phi- (\hat{m}+ \lambda h_1)^2) (\nabla J * h_1)\}\Big]\gamma_1 \, dx\, dt \\
 &- \int_{S \times \Omega} \Big[\nabla \cdot \{  2 \beta h_1 (1- \phi)(\nabla J * (\hat{m}+ \lambda h_1))+ 2 \beta (\hat{m} + \lambda h_1)(1-\phi)(\nabla J * h)\} \Big] \gamma_2 \, dx\, dt.
 \end{align*}
Therefore,
\begin{align*}
    &\frac{d}{d \lambda} \mathfrak{J}(\hat{m}+ \lambda h_1)\Big|_{\lambda = 0}\,\, \text{at}\,\, (\hat{m}, \hat{\phi})\\
    =& -\int_{S \times \Omega} \Big[ \partial_t h_1- \nabla \cdot \{ \nabla h_1+ 4\beta \hat{m}h_1 (\nabla J * \hat{m})- 2 \beta (\hat{\phi}- \hat{m}^2) (\nabla J * h_1)\}\Big]\gamma_1 \, dx\, dt \\
 &- \int_{S \times \Omega} \Big[\nabla \cdot \{ 2 \beta h_1 (1- \hat{\phi})(\nabla J * \hat{m})+ 2 \beta \hat{m} (1-\hat{\phi})(\nabla J * h_1)\} \Big] \gamma_2 \, dx\, dt.
\end{align*}

\begin{align*}
    \text{1st term:}&\, \int_{S \times \Omega} \partial_t h_1 \gamma_1\,dxdt = -\int_{S \times \Omega} \partial_t \gamma_1 h_1 \,dxdt+ \int_{\Omega} (h_1 \gamma_1)\Big|_{0}^T\, dx,\\
    \text{2nd term:}&\, \int_{S \times \Omega}  - \Delta h_1 \gamma_1\,dxdt= \int_{S \times \Omega}  - \Delta \gamma_1 h_1\,dxdt- \int_{S \times \partial\Omega} \gamma_1 \nabla h_1 \cdot \Vec{n}\,dsdt+ \int_{S \times \partial\Omega}  h_1\nabla \gamma_1 \cdot \Vec{n}\,dsdt,\\
     \text{3rd term:}&\, -\int_{S \times \Omega}  \nabla \cdot \{4\beta \hat{m}h_1 (\nabla J * \hat{m})\} \gamma_1\,dxdt= \int_{S \times \Omega}  4\beta \hat{m}h_1 (\nabla J * \hat{m})\cdot \nabla \gamma_1\,dxdt\\
     & \hspace{6cm 
     }\underbrace{- \int_{S \times \partial  \Omega}\gamma_1 \nabla \{ 4\beta \hat{m}h_1 (\nabla J * \hat{m})\}\cdot \Vec{n} \,dsdt }_{\substack{=0 \,\,\text{as}\,\, \hat{m} \text{ satisfies periodic boundary conditions}}},\\
     \text{4th term:}&\, \int_{S \times \Omega} \nabla \cdot \{ 2 \beta (\hat{\phi} - \hat{m}^2)( \nabla J * h_1)\} \gamma_1\,dxdt= -\int_{S \times \Omega} 2 \beta (\hat{\phi} - \hat{m}^2)( \nabla J * h_1)\cdot \nabla  \gamma_1\,dxdt\\
     & \hspace{6cm}+ \underbrace{\int_{S \times \partial  \Omega}\gamma_1 \nabla \{2 \beta (\hat{\phi} - \hat{m}^2)( \nabla J * h_1) \} \cdot \Vec{n}\,dsdt}_{\substack{=0 \,\,\text{as}\,\, \hat{m}, \hat{\phi} \text{ satisfy periodic boundary conditions}}}\\
     & \hspace{6cm} 
    =-\int_{S \times \Omega} 2 \beta (\hat{\phi} - \hat{m}^2)( \nabla J * \nabla \gamma_1)h_1\,dxdt,
    \end{align*}
    \begin{align*}
    \text{5th term:}&\, \int_{S \times \Omega} \nabla \cdot \{  2 \beta h_1 (1- \hat{\phi})(\nabla J * \hat{m})\} \gamma_2 \,dxdt= -\int_{S \times \Omega} 2 \beta h_1 (1- \hat{\phi})(\nabla J * \hat{m}) \cdot \nabla \gamma_2 \,dxdt\\
    & \hspace{6cm}+ \underbrace{\int_{S \times \partial \Omega} \gamma_2  \{ 2 \beta h_1 (1- \hat{\phi})(\nabla J * \hat{m}) \} \cdot \Vec{n}\, dsdt }_{\substack{=0 \,\,\text{as}\,\, \hat{m}, \hat{\phi} \text{ satisfy periodic boundary conditions}}}, \\
    \text{6th term:}&\, \int_{S \times \Omega} \nabla\cdot \{2 \beta \hat{m}(1-\hat{\phi})(\nabla J * h_1)\} \gamma_2 \,dxdt= -\int_{S \times \Omega} 2 \beta \hat{m}(1-\hat{\phi})(\nabla J * h_1) \cdot \nabla \gamma_2 \,dxdt\\
    & \hspace{6cm} +\underbrace{\int_{S \times \partial \Omega} \gamma_2\{2 \beta \hat{m}(1-\hat{\phi})(\nabla J * h_1)\}  \cdot \Vec{n}\, dsdt}_{\substack{=0 \,\,\text{as}\,\, \hat{m}, \hat{\phi} \text{ satisfy periodic boundary conditions}}} \\
    & \hspace{6cm}= -\int_{S \times \Omega} 2 \beta \hat{m}(1-\hat{\phi})(\nabla J *  \nabla \gamma_2)  h_1\,dxdt.
\end{align*}
Hence,
\begin{equation*}
    \frac{d}{d \lambda} \mathfrak{J}(\hat{m}+ \lambda h_1)\Big|_{\lambda = 0}\,\, \text{at}\,\, (\hat{m}, \hat{\phi})=0
\end{equation*}
implies
\begin{align}
    \int_{S \times \Omega}\{ \partial_t \gamma_1+\Delta \gamma_1-4\beta \hat{m} (\nabla J * \hat{m})\cdot \nabla \gamma_1+2 \beta (\hat{\phi} - \hat{m}^2)( \nabla J * \nabla \gamma_1)+2 \beta (1- \hat{\phi})(\nabla J * \hat{m}) \cdot \nabla \gamma_2\nonumber\\
    +2 \beta \hat{m}(1-\hat{\phi})(\nabla J *  \nabla \gamma_2)\}h_1\, dxdt+  \int_{\Omega} (h_1 \gamma_1)(T)\, dx+\int_{S \times \partial \Omega} h_1 \nabla \gamma_1 \cdot \Vec{n} \, dsdt =0. \label{adj 1}
\end{align}
Now, for all $h_1 \in C^\infty_0(S \times \Omega)$, \eqref{adj 1} becomes
\begin{align*}
     \int_{S \times \Omega}\{ \partial_t \gamma_1+\Delta \gamma_1-4\beta \hat{m} (\nabla J * \hat{m})\cdot \nabla \gamma_1+2 \beta (\hat{\phi} - \hat{m}^2)( \nabla J * \nabla \gamma_1)+2 \beta (1- \hat{\phi})(\nabla J * \hat{m}) \cdot \nabla \gamma_2\nonumber\\
    +2 \beta \hat{m}(1-\hat{\phi})(\nabla J *  \nabla \gamma_2)\}h_1\, dxdt=0.
\end{align*}
Since, $C^\infty_0(S \times \Omega)$ is dense in $L^2(\Omega)$, then we have 
\begin{align*}
    \partial_t \gamma_1+\Delta \gamma_1- 4\beta \hat{m} (\nabla J * \hat{m})\cdot \nabla \gamma_1+2 \beta (\hat{\phi} - \hat{m}^2)( \nabla J * \nabla \gamma_1)+2 \beta (1- \hat{\phi})(\nabla J * \hat{m}) \cdot \nabla \gamma_2\nonumber\\
    +2 \beta \hat{m}(1-\hat{\phi})(\nabla J *  \nabla \gamma_2)=0 \,\, \text{in}\, \, S \times \Omega.
\end{align*}
Next, choose $y_1 \in C^1( \overline{S \times \Omega})$ such that ${y_1}|_{S \times \partial \Omega}=0$, then 
\begin{align*}
    &\int_{\Omega} (h_1 \gamma_1)(T)\, dx=0 \,\,\text{in}\,\, \Omega,\\
    \Rightarrow &\,\,\,\gamma_1(T) =0 \,\,\text{in}\,\, \Omega.
\end{align*}
Finally, choose $y_1 \in C^1( \overline{S \times \Omega})$ to get that 
\begin{equation*}
    \nabla \gamma_1 \cdot \Vec{n}=0 \,\, \text{on}\,\, S \times \partial \Omega.
\end{equation*}
 Summarizing the above equalities yield
\begin{align*}
    \partial_t \gamma_1+\Delta \gamma_1- 4\beta \hat{m} (\nabla J * \hat{m})\cdot \nabla \gamma_1+2 \beta (\hat{\phi} - \hat{m}^2)( \nabla J * \nabla \gamma_1)+2 \beta (1- \hat{\phi})(\nabla J * \hat{m}) \cdot \nabla \gamma_2\nonumber\\
    +2 \beta \hat{m}(1-\hat{\phi})(\nabla J *  \nabla \gamma_2)=&0 \,\, \text{in}\, \, S \times \Omega,\\
    \nabla \gamma_1 \cdot \Vec{n} =& 0 \, \text{on}\, \, S \times \partial \Omega,\\
    \gamma_1(T,x) =&0 \,\, \text{in}\, \,  \Omega.
\end{align*}

Step 2: Next, we compute the Gâteaux derivative of the Lagrangian with respect to the variable $\phi$. The Gâteaux derivative of $\mathfrak{J}$ at $(\hat{m},\hat{\phi},\hat{\theta})$ in the direction $h_2$ is defined by

\begin{align*}
    &\mathcal{D}_\phi\mathfrak{J}(\hat{m}, \hat{\phi}, \hat{\theta})h_2= 0 \,\,\text{in other words}\,\,  \frac{d}{d \lambda} \mathfrak{J}(\hat{\phi}+ \lambda h_2)\Big|_{\lambda = 0}=0 \,\, \text{ for all}\,\, h_2 \,\, \text{with $h_2(0)=0$}.
\end{align*}
We can now compute the following: 
\begin{align*}
    &\frac{d}{d \lambda} \mathfrak{J}(\hat{\phi}+ \lambda h_2)\\
    =& \frac{d}{d\lambda}\Big[\frac{1}{2}\int_{S \times \Omega}|\hat{\phi} + \lambda h_2-\phi_d|^2\, dx\,dt+\frac{\delta}{2}\int_{S \times \Omega}|\theta|^2 \, dx\,dt \Big]\\
    &-\frac{d}{d \lambda} \int_{S \times \Omega}[\partial_t m- \nabla \cdot \{2 \beta( \hat{\phi}+\lambda h_2-m^2)(\nabla J * m)\}] \gamma_1\, dxdt\\
    &- \frac{d}{d\lambda}\int_{ S\times \Omega} [ \partial_t ( \hat{\phi} + \lambda h_2) - \nabla \cdot \{ \nabla (\hat{\phi}+ \lambda h_2)- 2 \beta m (1- \hat{\phi}- \lambda h_2)(\nabla J * m)\} + \alpha (1- \hat{\phi}-\lambda h_2)+ \theta]\gamma_2 \,dxdt\\
    =& \int_{ S\times \Omega}(\hat{\phi} + \lambda h_2-\phi_d) h_2 \, dx dt+\int_{ S\times \Omega} \nabla \cdot \{2 \beta h_2 (\nabla J * m) \} \gamma_1\,dxdt\\
    &- \int_{S \times \Omega} [\partial_t h_2 - \nabla \cdot \{ \nabla h_2- 2 \beta m (-h_2)( \nabla J * m)\}+ \alpha (-h_2)]\gamma_2\,dxdt 
    \end{align*}
\begin{align}
    &\frac{d}{d \lambda} \mathfrak{J}(\hat{\phi}+ \lambda h_2)\Big|_{\lambda = 0}\,\, \text{ at}\,\, (\hat{m}, \hat{\phi})\, \, =0 \nonumber\\
    \Rightarrow& \int_{ S \times \Omega}(\hat{\phi}- \phi_d)h_2\,dxdt+ \int_{ S\times \Omega} \nabla \cdot \{2 \beta h_2 (\nabla J * \hat{m}) \} \gamma_1\,dxdt\nonumber\\
    &- \int_{S \times \Omega} [\partial_t h_2 - \nabla \cdot \{ \nabla h_2- 2 \beta \hat{m} (-h_2)( \nabla J * \hat{m})\}+ \alpha (-h_2)]\gamma_2\,dxdt =0\nonumber\\
    \Rightarrow & \int_{ S \times \Omega}(\hat{\phi}- \phi_d)h_2\,dxdt+ \int_{ S\times \Omega} \nabla \cdot \{2 \beta h_2 (\nabla J * \hat{m}) \} \gamma_1\,dxdt+ \int_{ S\times \Omega} \partial_t \gamma_2 h_2\,dxdt- \int_{ \Omega}  ( \gamma_2 h_2)\Big|_{0}^T dx\nonumber\\
    &+\int_{ S\times\Omega } \Delta \gamma_2 h_2\,dxdt+\int_{ S\times\Omega } \nabla \cdot \{2 \beta \hat{m}h_2 (\nabla J * \hat{m})\} \gamma_2\,dxdt+\alpha \int_{ S\times \Omega} \gamma_2 h_2 \, dxdt\\
    &- \int_{ S \times \partial \Omega} h_2 \nabla \gamma_2 \cdot \Vec{n}\,dsdt=0\nonumber
    \end{align}
    \begin{align}
    \Rightarrow& \int_{ S \times \Omega}(\hat{\phi}- \phi_d)h_2\,dxdt- \int_{ S\times \Omega} 2 \beta h_2 (\nabla J * \hat{m}) \cdot \nabla \gamma_1 \, dxdt+ \int_{ S\times \partial\Omega}\gamma_1\{2 \beta h_2 (\nabla J * \hat{m}) \}\cdot \Vec{n}dsdt\nonumber\\
    &+  \int_{ S\times \Omega} \partial_t \gamma_2 h_2\,dxdt- \int_{ \Omega}  ( \gamma_2 h_2)\Big|_{0}^T dx+ \int_{ S\times\Omega } \Delta \gamma_2 h_2\,dxdt-\int_{ S\times\Omega }2 \beta \hat{m} h_2 (\nabla J * \hat{m})\cdot \nabla \gamma_2\,dxdt\nonumber\\
    & + \int_{ S \times \partial \Omega} \gamma_2 \{ 2 \beta \hat{m} h_2 (\nabla J * \hat{m}) \} \cdot \Vec{n} +\alpha \int_{ S\times \Omega} \gamma_2 h_2 \, dxdt- \int_{ S \times \partial \Omega} h_2 \nabla \gamma_2 \cdot \Vec{n}\,dsdt=0 \label{adj 2}
\end{align}

Following the same fashion as mentioned in Step 1, \eqref{adj 2} yields

\begin{align*}
    \partial_t \gamma_2 + \Delta \gamma_2 -2 \beta  (\nabla J * \hat{m}) \cdot \nabla \gamma_1 -2 \beta \hat{m} (\nabla J * \hat{m})\cdot \nabla \gamma_2+\alpha \gamma_2 =&-(\hat{\phi}- \phi_d)&& \text{in}\,\, S \times \Omega,\\
    \nabla \gamma_2 \cdot \Vec{n} =& 0 && \text{on}\, \, S \times \partial \Omega,\\
    \gamma_2(T,x) =&0 && \text{in}\, \,  \Omega.
\end{align*}

\begin{theorem}
    Let the assumptions (A1)--(A4) are satisfied and assume that $\hat{\theta}$ be an local optimal control of the optimal control problem $\mathcal{(OP)}$ with the associated state $(\hat{m}, \hat{\phi})= \mathcal{S}(\hat{\theta})$, then the following adjoint problem
    \begin{align}
\partial_t \gamma_1 + \Delta \gamma_1- 4\beta \hat{m} (\nabla J * \hat{m})\cdot \nabla \gamma_1+ 2 \beta (\hat{\phi}- \hat{m}^2)(\nabla J * \nabla \gamma_1)\,\,\,&\nonumber\\
+ 2 \beta  (1-\hat{\phi})(\nabla J * \hat{m})\cdot \nabla \gamma_2+ 2 \beta \hat{m}(1-\hat{\phi})(\nabla J * \nabla \gamma_2)=&0 && \text{in}\,\,S \times \Omega, \label{1}\\
    \partial_t \gamma_2 + \Delta \gamma_2 -2 \beta  (\nabla J * \hat{m}) \cdot \nabla \gamma_1 -2 \beta \hat{m} (\nabla J * \hat{m})\cdot \nabla \gamma_2+\alpha \gamma_2=&  \phi_d-\hat{\phi} && \text{in}\,\,S \times \Omega,\label{2}\\
    \nabla \gamma_1 \cdot \Vec{n}=0, \,\nabla \gamma_2 \cdot \Vec{n}=&0 &&\text{on} \,\, S \times \partial \Omega,\\
     \gamma_1 (T,x) = 0, \,& \gamma_1 (T,x)=0  &&\text{in}\,\, \Omega,\label{3}
\end{align}
has a unique solution $(\gamma_1, \, \gamma_2) \in  \mathcal{W}\times \mathcal{W}$. 

\end{theorem}
\begin{proof}
The existence of a solution to the adjoint problem \eqref{1}–\eqref{3} follows in a straightforward way, for instance by applying suitably the Galerkin method. In what follows next, we present a proof for the {\em a priori} estimate.
   \par Multiplying the equation \eqref{1} by $\gamma_1$ and integrating the resulting equalities over $ \Omega$, we obtain
\begin{align}
    &\frac{1}{2} \frac{d}{dt} ||\gamma_1||^2_{L^2(\Omega)}- ||\nabla \gamma_1||^2_{L^2(\Omega)}\nonumber\\
    =& \int_{\Omega} \Big\{4\beta \hat{m} (\nabla J * \hat{m})\cdot \nabla \gamma_1- 2 \beta (\hat{\phi}- \hat{m}^2)(\nabla J * \nabla \gamma_1)\nonumber\\
&- 2 \beta  (1-\hat{\phi})(\nabla J * \hat{m})\cdot \nabla \gamma_2- 2 \beta \hat{m}(1-\hat{\phi})(\nabla J * \nabla \gamma_2)\Big\}\gamma_1\,dx\nonumber\\
\leq & \Big\{4 \beta || \hat{m} (\nabla J * \hat{m})||_{L^\infty(\Omega)} || \nabla \gamma_1||_{L^2(\Omega)}  + 2 \beta ||(\hat{\phi}- \hat{m}^2)(\nabla J * \nabla \gamma_1)||_{L^\infty(\Omega)}\nonumber\\
& + 2 \beta ||(1-\hat{\phi})(\nabla J * \hat{m})||_{L^\infty(\Omega)} || \nabla \gamma_2||_{L^2(\Omega)} + 2 \beta  || \hat{m}(1-\hat{\phi})||_{L^2(\Omega)} ||\nabla J * \nabla \gamma_2||_{L^\infty(\Omega)} \Big\} ||\gamma_1||_{L^2(\Omega)}. \label{foc 4}
\end{align}

Again, multiplying \eqref{2} by $\gamma_2$ and integrating the resulting equalities over $ \Omega$, we obtain
\begin{align}
     \frac{1}{2} \frac{d}{dt} ||\gamma_2||^2_{L^2(\Omega)}-||\nabla \gamma_2||^2_{L^2(\Omega)}=& \int_{\Omega} \Big\{2 \beta \hat{m}(\nabla J * \hat{m}) \cdot \nabla \gamma_2+2 \beta  (\nabla J * \hat{m}) \cdot \nabla \gamma_1- \alpha \gamma_2\nonumber\\
     & -( \hat{\phi}- \phi_d)\Big\} \, \gamma_2 \,dx\nonumber\\
     \leq & \Big\{2 \beta || \hat{m}(\nabla J * \hat{m})||_{L^\infty(\Omega)} || \nabla \gamma_2||_{L^2(\Omega)} + \alpha ||\gamma_2||_{L^2(\Omega)}\nonumber\\
     & +2 \beta || \nabla J * \hat{m}||_{L^\infty(\Omega)} || \nabla \gamma_1||_{L^2(\Omega)}+ || \hat{\phi}- \phi_d||_{L^2(\Omega)}\Big\}||\gamma_2||_{L^2(\Omega)}. \label{foc 5}
\end{align}
Now, adding \eqref{foc 4} and \eqref{foc 5} and applying the Young's inequality, we obtain
\begin{align}
   & - \frac{d}{dt} (||\gamma_1||^2_{L^2(\Omega)}+ ||\gamma_2||^2_{L^2(\Omega)})+ (||\nabla \gamma_1||^2_{L^2(\Omega)}+ ||\nabla \gamma_2||^2_{L^2(\Omega)})\nonumber\\
    &\leq C_1 (||\gamma_1||^2_{L^2(\Omega)}+ ||\gamma_2||^2_{L^2(\Omega)}) + C_2 || \hat{\phi}- \phi_d||^2_{L^2(\Omega)}\nonumber,
\end{align}
where 
\begin{align*}
     C_1 = &C(1+ || \hat{m} (\nabla J * \hat{m})||^2_{L^\infty(\Omega)}+ ||(\hat{\phi}- \hat{m}^2)||^2_{L^\infty(\Omega)} || \nabla J ||^2_{L^2(\Omega)} + ||(1-\hat{\phi})(\nabla J * \hat{m})||^2_{L^\infty(\Omega)}\nonumber\\
     &+  || \hat{m}(1-\hat{\phi})||^2_{L^2(\Omega)} || \nabla J ||^2_{L^2(\Omega)} + || (\nabla J * \hat{m})||^2_{L^\infty(\Omega)}).
\end{align*}
Integrating the above inequality over $(t, T)$ and applying Grönwall's inequality, we obtain

\begin{align*}
    ||\gamma_1(t)||^2_{L^2(\Omega)}+ ||\gamma_2(t)||^2_{L^2(\Omega)}+ \int_{0}^{T}(||\nabla \gamma_1||^2_{L^2(\Omega)}+ ||\nabla \gamma_2||^2_{L^2(\Omega)})\,dt\\
    \leq C( || m_0||_{L^2(\Omega)} , || \phi_0||_{L^2(\Omega)}, || \theta||_{L^2(\Omega)}, T).
\end{align*}
Furthermore, we now see that
\begin{align*}
   & \int_{0}^{T} || \partial_t \gamma_1||_{H^{-1}(\Omega)} \,dt \\
   \lesssim&  \int_{0}^{T}  \Big\{ || \nabla \gamma_1 ||_{L^2(\Omega)} +  4 \beta || \hat{m} (\nabla J * \hat{m})||_{L^\infty(\Omega)} || \nabla \gamma_1||_{L^2(\Omega)}  + 2 \beta ||(\hat{\phi}- \hat{m}^2)(\nabla J * \nabla \gamma_1)||_{L^\infty(\Omega)}\nonumber\\
& + 2 \beta ||(1-\hat{\phi})(\nabla J * \hat{m})||_{L^\infty(\Omega)} || \nabla \gamma_2||_{L^2(\Omega)} + 2 \beta  || \hat{m}(1-\hat{\phi})||_{L^2(\Omega)} ||\nabla J * \nabla \gamma_2||_{L^\infty(\Omega)}\Big\} \, dt
\end{align*}
and 
\begin{align*}
    & \int_{0}^{T} || \partial_t \gamma_2||_{H^{-1}(\Omega)} \,dt \\
    \lesssim&  \int_{0}^{T}  \Big\{ || \nabla \gamma_2 ||_{L^2(\Omega)}+ 2 \beta || \hat{m}(\nabla J * \hat{m})||_{L^\infty(\Omega)} || \nabla \gamma_2||_{L^2(\Omega)} + \alpha ||\gamma_2||_{L^2(\Omega)}+2 \beta || \nabla J * \hat{m}||_{L^\infty(\Omega)} || \nabla \gamma_1||_{L^2(\Omega)}\\
    &+ || \hat{\phi}- \phi_d||_{L^2(\Omega)}\Big\}dt.
\end{align*}
The uniqueness of weak solutions is guaranteed by the linearity of the adjoint system.

\end{proof}


\begin{lemma}
Assume that conditions (A1)–(A4) are satisfied and let $\hat{\theta} \in \mathcal{U}_{ad} $  be an optimal control for the problem $( \mathcal{OP} ),$ with associated state $(\hat{m}, \hat{\phi}) = \mathcal{S}(\hat{\theta})$. Let $(\gamma_1, \gamma_2)$ denote the solution of the corresponding adjoint problem. Then the optimal control $\hat{\theta} $ satisfies the following variational inequality:
\begin{equation*}
    \int_{S \times \Omega} \big( \gamma_2 + \delta \hat{\theta} \big)(\theta - \hat{\theta}) \, dxdt \geq 0,
\end{equation*}
holding true all admissible controls $\theta \in \mathcal{U}_{ad} $.
\end{lemma}
\begin{proof}

Taking test functions $\gamma_1$ and $\gamma_2$ in the weak formulations of \eqref{lin 1} and \eqref{lin 2}, respectively, we obtain
\begin{align}
    \int_{S}\langle \partial_t \varphi_1, \gamma_1\rangle\,dt=&\int_{S \times \Omega}\nabla \cdot [ \nabla \varphi_1- 2 \beta (\varphi_2 - 2 \hat{m}\varphi_1)(\nabla J * \hat{m})-2 \beta (\hat{\phi}-\hat{m}^2)(\nabla J * \varphi_1)]\gamma_1 \, dxdt\nonumber\\ 
    =&-\int_{S \times \Omega} \nabla \varphi_1 \cdot \nabla \gamma_1 \, dxdt+ \int_{S \times \Omega}2 \beta (\varphi_2 - 2 \hat{m}\varphi_1)(\nabla J * \hat{m}) \cdot \nabla \gamma_1 \, dxdt\nonumber\\
    &+ \int_{S \times \Omega}2 \beta (\hat{\phi}-\hat{m}^2)(\nabla J * \varphi_1) \cdot \nabla \gamma_1 \, dxdt\label{4}
\end{align}
and 
\begin{align}
     \int_{S}\langle \partial_t \varphi_2, \gamma_2\rangle\,dt=&\int_{S \times \Omega}\Big\{\nabla \cdot[ \nabla \varphi_2- 2\beta \hat{m}(1- \hat{\phi})(\nabla J * \varphi_1)-2 \beta \varphi_1 (1- \hat{\phi})(\nabla J * \hat{m})\nonumber\\
    & \,\,+ 2 \beta \varphi_2 \hat{m}(\nabla J * \hat{m})]- \alpha \varphi_2 +h\Big\}\gamma_2\,dxdt\nonumber \\
    =& -\int_{ S\times \Omega}\nabla \varphi_2 \cdot  \nabla \gamma_2\, dxdt+\int_{ S\times \Omega}  2\beta \hat{m}(1- \hat{\phi})(\nabla J * \varphi_1) \cdot \nabla \gamma_2\,dxdt\nonumber\\
    &+ \int_{ S\times \Omega}2 \beta \varphi_1 (1- \hat{\phi})(\nabla J * \hat{m}) \cdot \nabla \gamma_2\,dxdt-\int_{ S\times \Omega}2 \beta \varphi_2 \hat{m}(\nabla J * \hat{m})\cdot\nabla \gamma_2\, dxdt \nonumber\\
    &-\alpha \int_{ S\times \Omega} \varphi_2 \gamma_2 \, dxdt+\int_{ S\times \Omega}  h \gamma_2 \, dxdt.\label{5}
\end{align}
Again, taking test functions $ \varphi_1$ and $\varphi_2$ for the weak formulations of \eqref{1} and \eqref{2}, respectively, we obtain

\begin{align}
    &\int_{S}\langle \partial_t\gamma_1, \varphi_1\rangle\,dt+ \int_{S \times \Omega}\Big\{\Delta \gamma_1- 4\beta \hat{m} (\nabla J * \hat{m})\cdot \nabla \gamma_1+ 2 \beta (\hat{\phi}- \hat{m}^2)(\nabla J * \nabla \gamma_1)\nonumber\\
&+ 2 \beta  (1-\hat{\phi})(\nabla J * \hat{m})\cdot \nabla \gamma_2+ 2 \beta \hat{m}(1-\hat{\phi})(\nabla J * \nabla \gamma_2)\Big\}\varphi_1\,dxdt=0\nonumber\\
\Rightarrow&-\int_{S}\langle \partial_t \varphi_1, \gamma_1,\rangle\,dt - \int_{ S\times \Omega} \nabla \gamma_1 \cdot \nabla \varphi_1\,dxdt- \int_{ S\times \Omega} 4\beta \hat{m} (\nabla J * \hat{m})\cdot \nabla \gamma_1 \varphi_1\,dxdt\nonumber\\
&+\int_{ S\times \Omega}2 \beta (\hat{\phi}- \hat{m}^2)(\nabla J * \nabla \gamma_1)\varphi_1\,dxdt+ \int_{ S\times 
\Omega}2 \beta  (1-\hat{\phi})(\nabla J * \hat{m})\cdot \nabla \gamma_2 \varphi_1\,dxdt\nonumber\\
&+\int_{ S\times \Omega}2 \beta \hat{m}(1-\hat{\phi})(\nabla J * \nabla \gamma_2)\varphi_1 \,dxdt=0\label{6}
\end{align}
and 
\begin{align}
    &\int_{S}\langle \partial_t\gamma_2, \varphi_2\rangle\,dt+ \int_{ S\times \Omega}\Big\{\Delta \gamma_2-2 \beta  (\nabla J * \hat{m}) \cdot \nabla \gamma_1-2 \beta \hat{m}(\nabla J * \hat{m}) \cdot \nabla \gamma_2+ \alpha \gamma_2\nonumber\\
     &\hspace{3cm}- (\hat{\phi}- \phi_d)\Big\}\varphi_2 \,dxdt=0\nonumber\\
    \Rightarrow&  -\int_{S } \langle \partial_t \varphi_2, \gamma_2\rangle \,dt- \int_{ S\times \Omega} \nabla\gamma_2 \cdot \nabla \varphi_2\,dxdt- \int_{ S\times \Omega} 2 \beta \hat{m}(\nabla J * \hat{m}) \cdot \nabla \gamma_2 \varphi_2 \,dxdt\nonumber\\
    &- \int_{ S\times \Omega} 2 \beta (\nabla J * \hat{m}) \cdot \nabla \gamma_1\,dxdt+ \alpha\int_{ S\times \Omega}\gamma_2 \varphi_2 \,dxdt- \int_{ S\times \Omega} (\hat{\phi}- \phi_d) \varphi_2\,dxdt=0.\label{7}
\end{align}
Now, for $h = \theta- \hat{\theta}$, equations \eqref{4}, \eqref{5}, \eqref{6} and \eqref{7} yield
\begin{align*}
    & \int_{S\times \Omega} (\hat{\phi}-\phi_d)\varphi_2 \, dx\,dt+ \delta\int_{S \times \Omega} (\theta- \hat{\theta})\hat{\theta} \,dx\,dt \geq 0\\
    \Rightarrow&\int_{S\times \Omega} (\gamma_2 + \delta\hat{\theta} )(\theta-\hat{\theta}) \,dxdt \geq 0.
\end{align*}
\end{proof}
Looking at the closing of this proof, it is worth noting that, due to the box constraints defining the admissible set $\mathcal{U}_{ad}$, the indicated variational inequality implies that the optimal control can be characterized pointwise as the $L^2(S \times \Omega)$ orthogonal projection of $-\frac{1}{\delta} \gamma_2$ on to $\mathcal{U}_{ad}$. This provides an explicit representation of the optimal control in terms of the adjoint variable.
\section{Outlook on possible extensions}
As anticipated already, our main interest lies {\em de facto} in formulating a boundary optimal control formulation in 3D as this fits best to the physical setting we have in mind. This work is just taking care of the theoretical foundation of a somewhat simpler problem.  From the mathematical analysis perspective, we expect to be able to draw some inspiration in this direction from the work \cite{fukao2017boundary}, among  other sources. Our long-term intention is to link the mathematical analysis of the problem that relies on PDE-based estimates with bounds on finite-volume schemes so that, at the end, we provide a computational efficient (convergent) framework to approximate numerically in 3D the solution to the target optimal control formulation. In other words, we wish to extend our work \cite{Nicklas_3D} to include realistic, ``controlled", evaporation effects.

Inspired by the RIMS talk of Prof. Shuji Yoshikawa (Hiroshima) (related to the works \cite{Yoshikawa0,Yoshikawa1}) made us realize that is in fact possible to employ for our setting the microforce balance concept by M. Gurtin \cite{Gurtin_microforces} (and collaborators like E. Fried cf. e.g. \cite{Fried_microforces}) together with the celebrated Coleman-Cole procedure to enrich our model with viscoelasticity properties for the mixture components, which  for organic solar cells morphologies are a mix of two different polymers immersed within a solvent. Consequently, such coupling in 3D between our current mass-balance model with viscoelasticity-like equations  is likely to lead to more realistic [and mathematically challenging] control problems.  Intuitively, we expect that the mechanical properties of the mixture components hinder the realization of optimal control plans that are purely based on evaporation model parameters;  this is yet to be investigated from both analysis and computational viewpoints. 

We are mentioning in passing that further optimal control ideas could makes sense to explore (at least mathematically): a regional control could focus on the local solvent behavior within a specific region of interest like in \cite{Anita}, or perhaps a sliding  mode control (in the spirit of \cite{Gabriela}) can be in principle posed to control morphology formation within a minimum time (which can perhaps be re-written in term of a minimum solvent boundary concentration for a (boundary) control setting with surface evaporation in 3D). All these are possible investigations paths that remain to be explored.

\section*{Acknowledgments}
Wenner Gren Foundation is acknowledged for their kind support via the project UPD2025-0126, which finances the activity of AK. AM thanks the Swedish Research Council (project nr. 2024-05606)  for a partial support. Additionally, AM is grateful for RIMS's hospitality during the open  RIMS Symposium {\em 
Evolution Equations and Related Topics--Abstract Structures and Versatilities} that took place in Kyoto in October 2025, under the organization of Profs.  T. Fukao (Ryukoku University) and T. Yokota (Tokyo University of Sciences).

\bibliographystyle{plain}
\bibliography{references}

\end{document}